\documentclass{article}

\newtheorem{theorem}{Theorem}[subsection]

\newtheorem{lemma}[theorem]{Lemma}

\newcommand{\qed}{\rule{1mm}{3mm}}     
\newenvironment{proof}{\vspace*{\parsep}\noindent {\bf proof:}}{\qed\\[1em]}
\usepackage{amsmath, amssymb}
\begin{document}

\title{CZF and Second Order Arithmetic}
\author{Robert S. Lubarsky\thanks{Thanks are due to the referee for the careful reading
and helpful suggestions received.}\\Dept. of Mathematical Sciences\\Florida Atlantic University\\Boca Raton, FL 33431, USA\\Robert.Lubarsky@alum.mit.edu}
\maketitle

\abstract{Constructive ZF + full Separation is shown to be equiconsistent with Second Order Arithmetic.}

\section{Introduction}

CZF, Constructive Zermelo-Fraenkel Set Theory, is an axiomatization of set theory in intuitionistic logic strong enough to do much standard mathematics yet modest enough in proof-theoretical strength to qualify as constructive. Based originally on Myhill's CST \cite{my}, CZF was first identified and named by Aczel \cite{aczel, ac2, ac3}. Its axioms are:
\begin{itemize}
\item Pairing: $\forall x, y \; \exists z \; \forall w \; w \in z \leftrightarrow (w=x \vee w=y)$
\item Union: $\forall x \; \exists y \; \forall z (z \in y \leftrightarrow \exists w (w \in x \wedge z \in w)$
\item Extensionality: $\forall x, y \; x=y \leftrightarrow \forall z (z \in x \leftrightarrow z \in y)$
\item Set Induction (Schema): $(\forall x ((\forall y \in x \; \phi(y)) \rightarrow \phi(x))) \rightarrow \forall x \; \phi(x)$
\item $\Delta_{0}$ Separation (Schema): $\forall x \; \exists y \; \forall z \; z \in y \leftrightarrow (z \in x \wedge \phi(z))$, for $\phi$ a $\Delta_{0}$ formula
\item Strong Infinity: $\exists x \; [(\emptyset \in x) \wedge (\forall y \in x \; (y \cup \{y\}) \in x) \wedge \forall z ((\emptyset \in z \wedge \forall y \in z \; y \cup \{y\} \in z) \rightarrow z \subseteq x)]$
\item Fullness (AKA Subset Collection): $\forall x, y \; \exists z \; \forall$ R (if R is a total relation from $x$ to $y$ then there is a total relation R' $\in$ z from $x$ to $y$ such that R' $\subseteq$ R)
\item Strong Collection (Schema): $\forall x (\forall y \in x \; \exists z \; \phi(y, z) \rightarrow \exists w \; (\forall y \in x \; \exists z \in w \; \phi(y, z) \wedge \forall z \in w \; \exists y \in x \; \phi(y, z)).$
\end{itemize}

Among Aczel's accomplishments was an interpretation of CZF, and various extensions thereof, in Martin-L\"of type theory, which established CZF as a predicative theory. As it turns out, CZF is proof-theoretically equivalent with Martin-L\"of type theory ML$_{1}$V, as well as KP (Kripke-Platek admissible set theory), ID (inductive definability), and a host of other identified theories. Such theories can arguably be considered to be the limit of predicative mathematics, allow for a philosophical justification as being constructive, and are in any case very much weaker proof-theoretically than ZF. For an overview, references, and some proofs, see \cite{rathjen}.

It is a natural enough question to ask about the strength of variants of CZF. Michael Rathjen conjectured that adding full Separation to CZF elevates the theory's strength from ID, which is a small fragment of Second Order Arithmetic, to full Second Order Arithmetic. In this note, this conjecture will be shown to be correct.

What is the value of such a result? It is certainly nicer if a theory is shown to be weak, in that one can then use that theory with fewer philosophical scruples. This is the case for instance in \cite{ac2}, see also \cite{rathjen}, where it is shown that adding certain choice principles to CZF does not change the proof-theoretic strength. Still, knowing that CZF + full Separation has the strength of Second Order Arithmetic tells us at least that the former theory is not predicative, providing a warning to the constructively-minded mathematician not to work within it. In addition, the exact determination of this strength provides independence results. For instance, Rathjen (\cite{rath2}, prop. 7.12 (ii)) shows that CZF + Power Set is proof-theoretically stronger than n$^{th}$ Order Arithmetic for all n, from which it follows that CZF, being much weaker, does not prove Power Set. Similarly, it follows from the result in this paper that even with the addition of full Separation Power Set does not follow. (The latter result is re-proven model-theoretically in \cite{lubarsky}.) None of this should be surprising. Both the main result of this paper and the consequences just cited were anticipated by Rathjen (and perhaps others), and the ideas contained in the proof to follow are themselves mostly a reworking of those found in \cite{aczel} and \cite{rathjen}. Still, at least now the results are formally established.

By way of additional background, recent work on CZF and other intuitionistic set theories, often involving categorical models, has been done by S. Awodey, C. Butz, N. Gambino, E. Griffor, A. Joyal, I. Moerdijk, E. Palmgren, A.K. Simpson, Th. Streicher, and M. Warren. For an excellent overview and references, see \cite{ac4} or \cite{simp}.

\section{Interpreting CZF + full Separation in Second Order Arithmetic}

First we will work within Second Order Arithmetic with the Axiom of Choice (AC$_{\omega}$) to prove the consistency of CZF + full Separation. The Axiom of Choice in this context is the assertion that $\forall n \; \exists X \; \phi(n, X) \rightarrow \exists X \; \forall n \; \phi(n, X_{n}),$ where $X_{n}$ is the n$^{th}$ slice of X according to some recursive coding scheme. It is an old observation that the consistency strength of Second Order Arithmetic is unchanged by the addition of AC$_{\omega}$; see for instance \cite{BFPS}.

NOTATION: Natural numbers will be used to stand for recursive functions, via some standard encoding. The variables used for numbers in this context will be $e, f, g, h,$ and variants thereof. It is assumed that the language in question has a two-place function symbol App, which on inputs e and x returns the result of applying the e$^{th}$ recursive function to x. For readability, this will be written as \{$e$\}(x). We will avail ourselves of $\lambda$-notation to describe recursive functions. This means that if some recursive procedure P is given for producing an integer P(x), depending uniformly on an input integer x, then $\lambda$x.P(x) will stand for an integer e such that \{e\}(x) = P(x). We also assume the choice of a distinguished recursive tupling function, $\langle \; \rangle$, with recursive arity function and recursive projection functions (-)$_{i}, i \in \omega$. Vector notation $\vec a$ will be used to denote tuples, i.e. $\vec a$ is an abbreviation for $\langle a_{0}, ... , a_{n} \rangle$. Concatenation of tuples is given by $\frown: \vec a \frown \vec b = \langle a_{0}, ... , a_{n}, b_{0}, ... , b_{m} \rangle$. If $\hat a$ is an integer then $\vec a \frown \hat a$ is taken as shorthand for $\vec a \frown \langle \hat a \rangle$, and similarly for $\hat a \frown \vec a$. If X is a collection of tuples, then $\vec a \frown X = \{ \vec a \frown \vec b \mid \vec b \in X \}$. Also, let $\vec a^{X}$ be \{ $\vec b \mid \vec a \frown \vec b \in$ X \}. If $\vec a = \langle a \rangle$, then $\langle a \rangle \frown X$ and $\langle a \rangle^{X}$ will often be abbreviated as $a \frown X$ and $a^{X}$ respectively. (Note that $\vec a^{X}$ might well be non-trivial even if $\vec a \not \in$ X, as X might contain only proper extensions of $\vec a$.)

The proof will be via a realizability interpretation. By way of terminology, we will refer to certain reals in the model of arithmetic as representing or being sets in the model of CZF. A real is itself a set of integers. In order to avoid confusion as to whether any given object is claimed to be a set of integers in the given model of arithmetic or to be a CZF-set, the word ``set" will be reserved exclusively for the latter, the former being referred to as collections of integers or, more simply, reals.

Under this interpretation, a set is given by a non-empty real S, consisting only of tuples, which forms a tree, and which moreover is well-founded:
$$\forall X [(X \subseteq S \wedge \forall \vec a \in S (\forall \vec a \frown \hat a \in S \; \vec a \frown \hat a \in X \rightarrow \vec a \in X)) \rightarrow X = S].
$$
(Any X satisfying the condition in the antecedent will be called {\it inductive}. Since inductivity is, strictly speaking, dependent on the choice of ambient real S, sometimes for clarity such an X will be described as {\it inductive relative to S}.) The idea is that the members of S are given by integers $a$ such that $\langle a \rangle \in$ S, $a$'s members are given by integers $b$ such that $\langle a, b \rangle \in$ S, and so on.
The property of being a set, notated as Set(S), is $\Pi_{1}^{1}$-definable.

\begin{lemma} If Set(S) and $\vec a \in$ S then Set($\vec a^{S}$).
\end{lemma}

\begin{proof}
Suppose X $\subseteq \vec a^{S}$ is inductive: $\forall \vec b \in \vec a^{S} (\forall \vec b \frown \hat b \in a^{S} \; \vec b \frown \hat b \in X \rightarrow \vec b \in X)$. We must show that X = $\vec a^{S}$.

Let X$^{+}$ = ($\vec a \frown X$) $\cup$ (S $\backslash$ $(\vec a \frown \vec a^{S})$). We claim that also X$^{+}$ is inductive (relative to S). To see this, suppose that, for a given $\vec b$, for all $\hat b$ such that $\vec b \frown \hat b \in S$, $\vec b \frown \hat b \in X^{+}$. If $\vec b$ does not extend $\vec a$ then $\vec b \in S \backslash (\vec a \frown \vec a^{S}) \subseteq X^{+}$. If $\vec b$ does extend $\vec a^{S}$, say $\vec b = \vec a \frown \vec c$, then, if $\vec c \frown \hat b \in \vec a^{S}, \; \vec b \frown \hat b \in S$. By the hypothesis on $\vec b$, $\vec b \frown \hat b \in X^{+}$, hence $\vec c \frown \hat b \in X$. Since X is inductive, $\vec c \in X$, which yields $\vec b \in X^{+}$. Hence X$^{+}$ is inductive. Since Set(S), X$^{+}$ = S.

To finish up, if $\vec b \in \vec a^{S}$, then $\vec a \frown \vec b$ is in S, and hence also in X$^{+}$. Since $\vec a \frown \vec b$ could not enter X$^{+}$ via the second clause (in X$^{+}$'s definition), it must have entered via the first: $\vec a \frown \vec b \in (\vec a \frown X)$. That means $\vec b \in X$, i.e. $\vec a^{S} \subseteq X$.
\end{proof}

The realizability relation is defined recursively as follows:
\begin{itemize}
\item $e \vdash X \in S \leftrightarrow e_{1} \vdash X = e_{0}^{S} \wedge \langle e_{0} \rangle \in S$
\item $e \vdash S=T \leftrightarrow (\forall \langle a \rangle \in S \; \{e_{0}\}(a) \vdash a^{S} \in T) \wedge (\forall \langle b \rangle \in T \; \{e_{1}\}(b) \vdash b^{T} \in S) $
\item $e \vdash \phi \vee \psi \leftrightarrow (e_{0} = 0 \wedge e_{1} \vdash \phi) \vee (e_{0} = 1 \wedge e_{1} \vdash \psi)$
\item $e \vdash \phi \wedge \psi \leftrightarrow (e_{0} \vdash \phi) \wedge (e_{1} \vdash \psi)$
\item $e \vdash \phi \rightarrow \psi \leftrightarrow \forall f \; (f \vdash \phi \rightarrow \{e\}(f) \vdash \psi)$
\item $e \vdash \forall X \phi(X) \leftrightarrow \forall X (Set(X) \rightarrow e \vdash \phi(X))$
\item $e \vdash \exists X \phi(X) \leftrightarrow \exists X (Set(X) \wedge e \vdash \phi(X))$
\end{itemize}
(As is standard, $\phi \leftrightarrow \psi$ is an abbreviation for $\phi \rightarrow \psi \wedge \psi \rightarrow \phi$, and $\neg \phi$ for $\phi \rightarrow \bot$, where it follows by omission from the above that nothing realizes $\bot$.)

Although intuitionistically the Levy hierarchy of $\Sigma_{n}$ and $\Pi_{n}$ formulas does not work as in the classical case, we will still have use of the classical Levy rank of a formula. In the following, the assertion ``$\phi$ is $\Sigma_{n}$ (resp. $\Pi_{n}$)" is to be understood as giving $\phi$'s classical rank, even when $\phi$ is intended and used as an intuitionistic formula. The reason for this is that the relation ``$e \vdash \phi$" is $\Sigma_{n+1}^{1}$ (resp. $\Pi_{n+1}^{1}$) for $\phi$ a $\Sigma_{n}$ (resp. $\Pi_{n}$) formula (classically), uniformly in $e$ and $\phi$. This fact follows easily from the inductive definition of $\vdash$.

We now show that under this interpretation all of the axioms of CZF are valid.
\begin{enumerate}
\item intuitionistic logic

That all of the (standard) realizability interpretations satisfy the axioms and
inference rules of intuitionistic logic is by now well understood. For examples
and details, see for instance \cite{beeson} or \cite{mccarty}.
(Thanks to the referee for relaying the information that in the
latter the proof of the closure lemma seems to be mistaken. A
correction is reported to appear in \cite{KGO}.)

\item equality axioms

Realizers need to be given for:
\begin{enumerate}
\item $\forall x \; x=x$

\item $\forall x, y \; x=y \rightarrow y=x$

\item $\forall x, y, z \; x=y \wedge y=z \rightarrow x=z$

\item $\forall x, y, z \; x=y \wedge z \in x \rightarrow z \in y$

\item $\forall x, y, z \; x=y \wedge x \in z \rightarrow y \in z.$

\end{enumerate}
\begin{enumerate}
\item Unraveling the definitions, we need to find a realizer e so that, for i = 0,1
$$\forall \langle a \rangle \in x \; \{e_{i}\}(a)_{1} \vdash \langle a \rangle^{x} = \{e_{i}\}(a)_{0}^{x}.
$$
With this in mind, let g be such that
$$\{g\}(f) \; = \; \langle \lambda x. \langle x,f \rangle, \lambda x. \langle x,f \rangle \rangle.
$$
Let e be a fixed point for g: e = \{g\}(e). e is as desired. (In clauses 9) and 10) below, e will be referred to as Id.)

\item $\lambda e. \langle e_{1}, e_{0} \rangle$ is as desired.
\end{enumerate}
(c)-(e) are of a similar flavor, and are left to the reader, who can also find proofs in \cite{mccarty}. Note that it follows by induction on formulas that for all formulas $\phi$ there is a realizer for $(x=y \wedge \phi(x)) \rightarrow \phi(y)$, with (c)-(e) being the base cases.

\item extensionality

A realizer for this follows readily from the equality axioms above and the definitions of ``$e \vdash S=T$" and ``$e \vdash X \in S$".

\item set induction

Given a formula $\phi$, we need a realizer for
$$\forall S (\forall T \in S \; \phi(T) \rightarrow \phi(S)) \rightarrow \forall S \phi(S).
$$
Let e realize the antecedent. We need a realizer h(e) for the conclusion, recursively and uniformly in e; then $\lambda$e.h(e) would be the desired realizer. Let h(e) be \{e\}($\lambda$ x. h(e)). By the Recursion Theorem, h(e) is a recursive function, and by its definition is uniform in e. To see that $h(e) \vdash \forall S \phi(S)$, let S be a set; we need to show that $h(e) \vdash \phi(S)$. Let X be $\{ \vec a \in S \mid h(e) \vdash \phi(\vec a^{S})\}.$ By the assumption on e and definition of h(e), if $\forall \vec a \frown \hat a \in S \; \vec a \frown \hat a \in X$, then $\vec a \in X$. By the well-foundedness of S, X=S. Hence $\langle \; \rangle \in$ X, and $h(e) \vdash \phi(\langle \; \rangle^{S})$.

\item pairing

Given S and T, let U be $(0 \frown S) \cup (1 \frown T)$. U is as desired.

\item union

Given S, let T be $\{ \vec a \mid \exists a \; a \frown \vec a \in S \} \; \cup \; \{\langle \; \rangle \}$. T is as desired. (The reader might wonder why we have to throw the empty sequence into T. Recall that, in order to represent a set, T must be non-empty, by definition. If S represents the empty set, i.e. if S = $\{ \langle \; \rangle \}$, then $\{ \vec a \mid \exists a \; a \frown \vec a \in S \}$ would be the empty set itself. In this case T as actually defined would consist solely of $\langle \; \rangle$, making $\bigcup \emptyset = \emptyset$, as desired.)

\item strong infinity

We will define S$_{n}$ inductively on n. Let S$_{0}$ be $\{ \langle \; \rangle \}$. Given S$_{m} (m<n) $ let S$_{n}$ be $\{ m \frown \vec a \mid m<n \wedge \vec a \in S_{m} \}$. Let S$_{\omega}$ be $\{ n \frown \vec a \mid n \in \omega, \; \vec a \in S_{n} \}$. S$_{\omega}$ is as desired.

\item full separation

Given a set S and a formula $\phi$, we need to show the existence of $\{ a \in S \mid \phi(a) \}$. Let S$_{\phi}$ be $\{ \langle \langle f, a_{0} \rangle, a_{1}, ... , a_{n} \rangle \mid \langle a_{0}, a_{1}, ... , a_{n} \rangle \in S \wedge f \vdash \phi(a_{0}^{S}) \}$. S$_{\phi}$ exists, by the definability of $\vdash$. It is also a tree, almost, lacking merely the empty sequence. So let Set$_{S, \phi}$ be S$_{\phi}$ $\cup \; \{ \langle \; \rangle \}$. We claim that Set$_{S, \phi}$ is as desired.

To see that Set$_{S, \phi}$ is a set, it remains only to check that it is well-founded. To this end, suppose X $\subseteq$ Set$_{S, \phi}$ is inductive. It is easy to verify that for any $\vec a$ $\vec a^{X}$ is an inductive subset of $\vec a^{Set_{S, \phi}}$. Moreover, if $\vec a \not = \langle \; \rangle$ then there is a $\vec b$ such that $\vec a^{Set_{S, \phi}}$ = $\vec b^{S}$ (For $\vec a = \langle \langle f, a_{0} \rangle, a_{1}, ... , a_{n} \rangle$, let $\vec b = \langle a_{0}, a_{1}, ... , a_{n} \rangle$.). By the lemma earlier in this section, $\vec b^{S}$ is a set (i.e. is well-founded), so $\vec a^{X} = \vec b^{S}$ and $\vec a^{X} = \vec a^{Set_{S, \phi}}$. This shows that S$_{\phi} \subseteq$ X. As a final step, since X is inductive, Set$_{S, \phi} \subseteq$ X.

It remains only to find a realizer for
$$ \forall x \; (x \in Set_{S, \phi} \leftrightarrow (x \in S \; \wedge \; \phi(x)))
$$
which is independent of S and Set$_{S, \phi}$. (The independence is necessary by
the $\forall$ and $\exists$ clauses in the definition of $\vdash$, which do not
allow the realizer to access the sets chosen.) Choose any set x. Going from left
to right, suppose e $\vdash ``x \in Set_{S, \phi}"$. That means that e$_{1}
\vdash ``x = e_{0}^{Set_{S, \phi}}"$ and $\langle e_{0} \rangle \in Set_{S, \phi}$. By the definition of Set$_{S, \phi}$,
$\langle e_{01} \rangle \in$ S, $e_{01}^{S} = e_{0}^{Set_{S, \phi}}$, and
$e_{00} \vdash \phi(e_{01}^{S})$. This yields that $\langle e_{01}, e_{1}
\rangle \vdash ``x \in S"$. Furthermore, from $e_{00} \vdash \phi(e_{0}^{Set_{S, \phi}})$ and e$_{1} \vdash ``x = e_{0}^{Set_{S, \phi}}"$, a realizer g for $\phi(x)$ can be computed (uniformly in $e_{00}$ and $e_{1}$). The desired realizer for the right hand side is then $\langle \langle e_{01}, e_{1} \rangle, g \rangle$. The other direction is similar.

\item subset collection

We will prove Fullness, which means finding an integer e such that
$$\forall S, T \; \exists C \; e \vdash ``C \subseteq TotRel(S,T) \wedge \forall U \in TotRel(S,T) \; \exists V \in C \; V \subseteq U"
$$
where TotRel(S,T) is the collection of total relations from S to T. The choice of
e will be facilitated by considering the kind of C's we will ultimately be working
worth (although e must not depend on the choice of C!). To motivate the choice of
C's, suppose we have a realizer for U being such a total relation:
f $\vdash \forall x \in S \; \exists y \in T \; \langle x, y \rangle \in U$
(neglecting here that every element in U must also come from S $\times$ T).
Unraveling the definitions yields
$$\forall x, g (g \vdash ``x \in S" \; \rightarrow \; \exists y \{f\}(g)_{01} \vdash ``y = \{f\}(g)_{00}^{T}"
\wedge \{f\}(g)_{11} \vdash ``\langle x,y \rangle = \{f\}(g)_{10}^{U}").
$$
Letting Id be such that Id $\vdash \forall x \; x=x$, and $a$ so that $\langle a \rangle \in$ S, unraveling the definitions shows that $\langle a,Id \rangle \vdash a^{S} \in$ S. Instantiating x with $a^{S}$ and g with $\langle a,Id \rangle$ produces
$$\{f\}(\langle a,Id \rangle)_{11} \vdash ``\langle a^{S},y \rangle = \{f\}(\langle a,Id \rangle)_{10}^{U}".
$$
Let
$$f \upharpoonright S \times T = \{ \vec a \in U \mid \exists \langle a \rangle \in S \; a_{0} = \{ f \} (\langle a,Id \rangle)_{10} \} \cup \{ \langle \; \rangle \}.
$$

From the definition, it would seem that $f \upharpoonright S \times T$ depends on U
as well as f and S, and not on T,
but actually that's not the case, which is important in what follows
(hence mention of U is suppressed in the
notation). That $f \upharpoonright S \times T$ does not depend on U
can be seen from the role of f. Since f $\vdash \forall x \in S \; \exists y \in T \; \langle x, y \rangle \in
U$, given f, S, and T, we can determine U, or at least that part of U
relevant to the definition of $f \upharpoonright S \times T$. (In a
little detail, feed $\langle a, Id \rangle$ to \{f\}, as $a$
ranges through the members of S. \{f\} will produce the code for
the corresponding $y$ in T, as well as that for $\langle a, y
\rangle$ in U. To get more information about U, meaning to get
codes for members of members, unravel the given $a$ and the produced
$y$.) So if the same f realizes that some V $\not =$ U
is also a total relation (from S to T), which is possible, the same $f
\upharpoonright S \times T$ would be produced using V instead of
U.

Let
$$^{S}T = \{ f \frown \vec a \mid \exists U \; f \vdash ``U \in TotRel(S,T)" \wedge \vec a \in f \upharpoonright S \times T \} \cup \{ \langle \; \rangle \}.
$$

$^{S}T$ will be the desired C. Toward this end, first we must verify
that $^{S}T$ is a set. The most difficult part of this is that it be
well-founded. A member of $^{S}T$ is given by an f realizing that
some U is a total relation. However, there could be many different
U's realized as a total relation by the same f. If we were to
stick all those different U's together, the result might not be
well-founded. However, since $f \upharpoonright S \times T$
doesn't depend on U, this is not a problem. The rest of showing
that $^{S}T$ is a set is left to the reader.

Next we need a realizer for $^{S}T$ being a set of total
relations. A member of $^{S}T$ is given by a realizer f that some
U is a total relation; that same f will work as such a realizer
for the member of $^{S}T$ named by f (essentially $f
\upharpoonright S \times T$). Also, given a U realized by f to be a
total relation, we must produce a member V of $^{S}T$ and a
realizer that V $\subseteq$ U. This V will be $f \upharpoonright S \times
T$;
its name in $^{S}T$ will be exactly U's
realizer f; and it is easy to realize the necessary set inclusion, as $f \upharpoonright S \times
T$ is literally a subset of U.

We leave to the reader the piecing together of the strands above for the
determination of the realizer e showing that $^{S}T$ is full. The
fact that the procedure just sketched will work for any S and T
shows that e can be chosen independently of S and T.
Observe that not only is $^{S}T$ full, it is also the set of functions from S to T.

\item strong collection

Suppose
$$e \vdash \forall X \in S \; \exists Y \; \phi(X,Y).
$$
This implies that
$$\forall \langle a \rangle \in S \; \exists Y \; \{e\}(\langle a, Id \rangle) \vdash \phi(a^{S}, Y),
$$
where Id $\vdash$ $\forall x \; x=x$. By the Axiom of Choice, for each such $a$ let $Y_{a}$ be such a Y. Let Z be
$$\{ \langle a, a_{0}, ... , a_{n} \rangle \mid \langle a_{0}, ... , a_{n} \rangle \in Y_{a} \} \cup \{ \langle \; \rangle \}.
$$
Z is as desired.

\end{enumerate}

\section{Interpreting Second Order Arithmetic in CZF + full Separation}
Within CZF + full Separation, let the natural numbers be given by the set posited
by the Strong Infinity Axiom, and let the sets of natural numbers be all subsets
of said set. This is a model of Intuitionistic Second Order Arithmetic. It is an
already established result that the consistency strength of Second Order
Arithmetic is unchanged by going from intuitionistic to classical logic (see for instance \cite{tr}, where
this is shown via a double-negation translation).

\end{document}